\pdfoutput=1
\documentclass{amsart}

\usepackage{fancyhdr}
\usepackage{graphicx}
\usepackage{pb-diagram,pb-xy}
\usepackage[all]{xy}
\usepackage[utf8]{inputenc}
\usepackage{mathrsfs}
\usepackage{amssymb}

\newtheorem{theorem}{Theorem}[section]
\newtheorem{lemma}[theorem]{Lemma}
\newtheorem{prop}[theorem]{Proposition}
\newtheorem{cor}[theorem]{Corollary}
\theoremstyle{definition}
\newtheorem{definition}[theorem]{Definition}

\newtheorem{remark}[theorem]{Remark}

\DeclareMathOperator{\Ext}{Ext}

\DeclareMathOperator{\Tor}{Tor}

\numberwithin{equation}{section}

\usepackage{hyperref}
\usepackage{comment}
\begin{document}

\title{Macdonald dual in the semilocal case and applications to local homology and local cohomology}


\author{A. L. Nascimento Filho}

\email{antonival@ufersa.edu.br}







\begin{abstract}


This paper extends Macdonald’s duality theory and Cuong–Nam’s local homology theory — which is in some sense dual to the local cohomology theory of A. Grothendieck [Gro1] — to semilocal noetherian rings equipped with the $J$-adic topology. By overcoming the key structural obstructions of the semilocal setting via Zöschinger’s structure theory and semilocal Matlis duality, this work establishes topological limit duality isomorphisms and links local homology with local cohomology. Ultimately, this framework provides the foundational machinery required to extend the classical results of Macdonald [Mac] and Cuong–Nam [C-N1] to the semilocal setting.

\end{abstract}

\maketitle

\section*{\textbf{Introduction}}

In this paper we study the Macdonald Dual of a $R$-module in context  the $R$ is semilocal. The Macdonald dual of a topological $R$-module $M$, in which case $R$ is local, is given by $M^*=Hom_R(M,\,E)$ where $E=E(R/\mathfrak{m})$ is the injective envelope of residual field and we consider only continuous homomorphisms. The study in this local case was well developed by I. G. Macdonald in his paper \textit{Duality over complete local rings} [Mac] of 1964. The Matlis Dual is given by $D(M)=Hom_R(M,\,E)$ where considers all homomorphisms, including those that are not continuous. A. Ooishi, in his paper \textit{Matlis duality and the width of a module} [Ooi] of 1976, studied such a dual in the semilocal case, where $E=E(R/\mathfrak{m})$ is replaced by $E(\bigoplus_{i=1}^nR/\mathfrak{m}_i)$. Here we study the Macdonald dual in which case $E$ is replaced in the same way.

The first section presents the linearly topologized modules, the linearly compact modules and some topological properties of these modules. Furthermore, we consider the modules to be linearly compact on a semilocal noetherian ring and present results involving these modules. 

In the second section, we begin to consider the complete ring with the $\mathtt{J}$-adic topology (where $\mathtt{J}$ is the Jacobson radical of the $R$) and from that some results are acquired. 

 In addition, Macdonald Dual in the semilocal case is studied and many results on this dual are achieved, culminating in the Proposition 2.6. We will apply the results obtained in the first two sections to local homology theory.

A local homology theory for artinian modules was presented for N. T. Cuong and T. T. Nam in the paper \textit{The $I$-adic completion and local homology for artinian modules} [C-N2] in 2001, where was defined the \textit{i}th \textit{local homology module} $H_i^I(M)$ of an $R$-module $M$
with respect to the ideal $I$ by \[H^I_i(M)=\varprojlim_tTor_i^R(R/I^t,\,M)\]where $R$ is local noetherian. This definition is in some sense dual to Grothendieck's  definition of local cohomology modules for
finitely generated modules over noetherian rings [Gro1] where the re was defined the \textit{i}th \textit{local cohomology module} $H^i_I(M)$ of an $R$-module $M$
with respect to the ideal $I$ by \[H^i_I(M)=\varinjlim_tExt_R^i(R/I^t,\,M)\]where $R$ is local noetherian.

Still in [C-N2] was showed many theorems properties over local homology modules and that $H_i^I(M)\cong L_i^I(M)$ when $M$ is artinian, where $L_\bullet^I(-)$ is the left derived functor of the $I$-adic completion functor $\Lambda_I(-)=\varprojlim_t(R/I^t\otimes_R-)$.

Posteriorly, in 2008, Cuong and Nam developed this theory by replacing the Artinian module hypothesis with a linearly compact module in paper \textit{A local homology theory for linearly compact modules} [C-N1] generalizing results of [C-N2] and presenting some results of duality and reflexivity of certain modules in relation to the \textit{Macdonald Dual}, studied by I. G. Macdonald in the paper \textit{Duality over complete local rings} [Mac] of 1962,

In [C-N1] many results of [Mac] was utilized and, naturally, the hypothesis $R$ local was required. Here, we generalize the paper [C-N1] for semilocal case. 

In the section three, we recall the definition of \textit{i}th local homology module, 
generalize for semilocal case one result that establishes a criterion for a module finitely generated over a noetherian semilocal ring to be linearly compact, a result that relates local homology module of a $R$-module $M$ and local homology module of $\Gamma_I(M)$, the $I$-torsion functor, is showed, 
 some results on the noetherianity of local homology modules of compact linearly semi-discrete modules are presented,  in addition to results involving Macdonald duality between local homology modules and local cohomology modules.

Finally, section four contains results involving the noetherian  dimension \newline$Ndim (-)$ and the Krull dimension $dim (-)$ of linearly compact semi-discrete modules, including generalizations of well-known results in Grothendieck's local cohomology theory, a criterion for vanishing certain modules of local cohomology and some results of artinianity of certain modules of local cohomology of linearly compact semi-discrete modules.

We build the necessary framework to extend the results of \textit{Duality over complete local rings} [Mac] and \textit{A local homology theory for linearly compact modules} [C-N1] to the semilocal setting, providing detailed proofs specifically for those results where the original local techniques break down. For the omitted results, the proofs are analogous to the local case, obtained by replacing the classical local tools with their corresponding semilocal counterparts established in this paper. For this reason, we cite the original results as references throughout.

We suggest that the reader see chapters 2 and 3 of [Mac] if they are not familiar with the basic results on linearly topologized modules and linearly compact modules.

Throughout the text, R is a commutative ring with identity.
\section{Linearly Compact Modules Over a Semilocal Ring}

\begin{definition}

	Let $R$ be a topological ring and $M$ be a topological $R$-module. 

	$M$ is said to be \textit{\textbf{linearly topologized}} if $M$ has a fundamental system of neighborhoods
	of zero $\mathcal{M}$ consisting of submodules (for example, a discrete topological $R$-module is linearly topologized). The elements of $\mathcal{M}$ are open submodules of $M$, and $\mathcal{M}$ satisfies the following conditions:

	\item \textbf{(1.1.1)} Given $N_1$ and $N_2$ in $\mathcal{M}$ there is an $N_3$ in $\mathcal{M}$ such that $N_3\subseteq N_1\cap N_2$;
	
	\item \textbf{(1.1.2)} Given $x\in M$ and $N\in\mathcal{M}$ there is a nucleus $U$ of $R$ such $Ux\subseteq N$.

	Conversely, if $M$ is an R-module and $\mathcal{M}$ is a family of submodules of $M$, then $\mathcal{M}$ is
	a fundamental system of neighborhoods of zero on $M$ provided $(1.1.1)$ and $(1.1.2)$ are satisfied, and the topology of $M$ is then uniquely determined.
\end{definition}
\begin{remark}
	
	If $M$ is a linearly topologized  $R$-module, $N$ a submodule of $M$, then $N$ (with the induced topology) and $M/N$ (with the quotient topology) are linearly topologized $R$-modules. $M/N$ is Hausdorff if and only if $N$ is closed in $M$.
	
\end{remark}

\begin{definition}
	
	Let $R$ be a topological ring. From now on we shall have to deal exclusively with Hausdorff linearly topologized $R$-modules, and henceforth ``$R$-module'' shall mean ``Hausdorff linearly topologized $R$-module''. An $R$-module $M$ is \textit{\textbf{linearly compact}} if $M$ has the following property: if $\mathcal{F}$ is a family of closed cosets (i.e. cosets of closed submodules) in $M$ which has the finite intersection property, then the cosets in $\mathcal{F}$ have a non-empty intersection.
\end{definition}
 Every artinian $R$-module is linearly compact with the discrete topology [Zel] Proposition 5.

\begin{remark} \text{ }
	
	\subitem (i) If $M$ is an $R$-module, $N$ a linearly compact submodule of $M$, then $N$ is closed in $M$.
	\subitem (ii) If $M$ is a linearly compact $R$-module, $N$ a closed submodule of $M$, then $N$ is a linearly compact.
	\subitem (iii) 	If $M,N$ are $R$-modules, $M$ linearly compact, $f:M\rightarrow N$ a continuous homomorphism, then $f(M)$ is linearly compact and hence is a closed map. 
	\subitem (iv) If $M$ is an $R$-module, $N$ a closed submodule of $M$, then $M$ is linearly compact if and only if $N$ and $M/N$ are linearly compact. 
	
\end{remark}

Let $(R,\,\mathfrak{m}_1,\,...,\,\mathfrak{m}_n)$ be a noetherian semilocal ring and $\mathtt{J}=\bigcap_{i=1}^n\mathfrak{m}_i$ its Jacobson ideal.

Considering $R$ with the $\mathtt{J}$-adic topology it becomes a Hausdorff topological ring. Let $\widehat{R}=\varprojlim\dfrac{R}{\mathtt{J}^n}$ be the $\mathtt{J}$-adic completion of $R$: $\widehat{R}$  is a complete noetherian semilocal ring, and its topology is the $\widehat{\mathtt{J}}$-adic topology, where $\widehat{\mathtt{J}}$ is the extension of $\mathtt{J}$ to $\widehat{R}$ via the natural homomorphism $\eta:R\rightarrow\widehat{R}$.

If $N$ is a closed $R$-submodule of $\widehat{R}$ then $N$ is an ideal of $\widehat{R}$. 
Furthermore, all ideals of $\widehat{R}$ are closed in $\widehat{R}$.

Let $M$ be an $R$-module
. If each element of $M$ is annihilated by a power of $\mathtt{J}$, we say that $M$ is $\mathtt{\textbf{J}}$-\textit{\textbf{torsion}}. 
If $M$ is a $\mathtt{J}$-torsion $R$-module then, $M$ have structure of $\widehat{R}$-module.

The injective envelope $E(R/\mathtt{J})$ of $R/\mathtt{J}$ will play an important role in what follows. We shall denote $E(R/\mathtt{J})$ by $R^*$.

\begin{remark}

	\textup{$R^*$ is an artinian $\mathtt{J}$-torsion $R$-module. Indeed, by Chinese Remainder Theorem, we have $$\dfrac{R}{\mathtt{J}}\cong\dfrac{R}{\mathfrak{m}_1}\oplus...\oplus\dfrac{R}{\mathfrak{m}_n}.$$ Hence $R^*$ is artinian (see [B-S] $10.1.4$ and $10.2.5$). Now, since $R^*$ is artinian, we have that 
	there exists $t\in\mathbb{N}$ (the set of natural numbers)
	such that $\mathtt{J}^tx=0$, with $x\in R^*$.}

	\textup{Furthermore $R^*$ (with the discrete topology) is a linearly topologized $R$-module. In fact, is an $\widehat{R}$-module and as such its lattice of $\widehat{R}$-submodules is identical with its lattice of $R$-submodules.
		From [Mac] $3.10$ the only linear topology that $R^*$ can carry is the discrete topology satisfying $(1.1.1)$. Since $R^*$ is $\mathtt{J}$-torsion the condition $(1.1.2)$ is satisfied and therefore $R^*$ (with the discrete topology) is a linearly topologized $R$-module.}
\end{remark}

\begin{remark}
	
 A $R$-module $M$ is \textbf{\textit{semi-discrete}} if every submodule of $M$ is closed.
Discrete $R$-modules are semi-discrete. The class of semi-discrete linearly
compact modules contains all artinian modules. Moreover, when $R$ is complete, a finitely generated $R$-module $M$ is  semi-discrete linearly
compact. Indeed, since $ R\cong\varprojlim
R/\mathtt{J}^t$ and each $R/\mathtt{J}^t$ is artinian, $R$ is semi-discrete
linearly compact with the $\mathtt{J}$-adic topology. As $M$ is a finitely generated $R$-module we have
an open continuous epimorphism $R^t\rightarrow M$. 
Therefore $M$ is semi-discrete and linearly compact
(Proposition 2.2 also provides this). In consequence the class of semi-discrete linearly compact modules contains also all finitely
generated modules over a complete ring. However, there are semi-discrete linearly
compact $R$-modules which are neither artinian nor finitely generated. Examples for this conclusion can be found in [C-N1] 7.3 (i).
\end{remark}


\begin{theorem} 
	 Let $M$ be a linearly compact semi-discrete $\mathtt{J}$-torsion $R$-module. Then $M$ is artinian.
\end{theorem}

\begin{proof}

	By main Theorem in [Z{\"o}s], there is a finitely generated submodule $B$ of $M$ such that $M/B$ is an artinian module. Since $M$ is $\mathtt{J}$-torsion, $B$ is also, and $IB = 0$ for some ideal $I$ power of $\mathtt{J}$.
	This implies that $B$ is a finitely generated module over $R/I$, which is an artinian
	ring. Thus, $B$ is an artinian $R/I$-module, and hence $B$ is an artinian
	$R$-module.
	
	Therefore, $M$ is artinian.

\end{proof}
The technique used by Macdonald in this result above does not apply to our case, as the fact that $R/\mathfrak{m}$ is a field is strongly used.


	If $M$ is an $R$-module, a \textbf{character} of $M$ is a continuous homomorphism from $M$ into $R^*$.
	\begin{prop}

		Let $M$ be an $R$-module. Then $M$ has enough characters, that is given any $x\neq0$ in $M$ there is a character $u$ of $M$ such that $u(x)\neq0$.
	\end{prop}
	\begin{proof} Since $M$ is Hausdorff it has an open submodule $U$ which does not contains $x$. Let $\pi_U$ be the canonical projection from $M$ into $M/U$ and let $\overline{x}$ be the image of $x$ in $M/U$. 
	Since $(0:_R\overline{x})$ is a proper ideal of $R$, we have that

 $(0:_R\overline{x})\subseteq\mathfrak{m}_j$ for some $j$. Define the homomorphism
		\begin{eqnarray*}
			\psi:<\overline{x}>&\rightarrow&\bigoplus_{i=1}^{n}R/\mathfrak{m}_i \\ a\cdot\overline{x}&\mapsto&(0+\mathfrak{m}_1,\ldots,\,a+\mathfrak{m}_j,\ldots,\,0+\mathfrak{m}_n). 
		\end{eqnarray*}	
		Note that $\psi$ is well defined. Indeed,
		
		$b\cdot\overline{x}=a\cdot\overline{x}\Rightarrow(a-b)\cdot\overline{x}=0\Rightarrow a-b\in(0:_R\overline{x})\Rightarrow a-b\in\mathfrak{m}_j\\\Rightarrow a+\mathfrak{m}_j=b+\mathfrak{m}_j.$
		
		Now, since $\bigoplus_{i=1}^{n}R/\mathfrak{m}_i\cong R/\mathtt{J}$ and denoting $i:R/\mathtt{J}\hookrightarrow E(R/\mathtt{J})=R^*$, we have the commutative diagram
		$$\xymatrix{
			&R^*& \\
			0\ar[r]&<\overline{x}>\ar[r]_{j}\ar[u]^{i\circ\psi}&M/U\ar@{-->}[ul]_{\overline{u}}\\
		}$$
		where $j$ is the inclusion and there is $\overline{u}$ since $R^*$ is injective. Note that $\overline{u}$ is continuous since $M/U$ is discrete.
		Set $u=\overline{u}\circ\pi_U$. We have $u\in Hom_R(M,\,R^*)$ and $$u(x)=\overline{u}\circ\pi_U(x)=\overline{u}(\overline{x})=\overline{u}(j(\overline{x}))=i(\psi(\overline{x})).$$ Since $\psi(\overline{x})\neq0$, $u(x)\neq0$.  
	\end{proof}	
	As in our case there can be more than one maximal ideal, it was necessary to create the map $\psi$ and use the Chinese Remainder Theorem.
	
	
	Since the theory of essential extensions and injective envelopes is valid in any abelian category which satisfies (AB5) and has a generator (see [Gab]) it follows that the dual of this theory is valid in $\mathcal{C}$. The essential features are as follows:

Let $M$ be a linearly compact $R$-module. A \textit{covering} of $M$ is a continuous epimorphism $f:Q\rightarrow M$ where $Q$ is a linearly compact $R$-module. $f$ is an \textit{essential covering} of $M$ if for every proper closed submodule $N$ of $Q$ we have $N+ker\,f\subsetneq Q$.

\begin{definition}\textit{}
	
	(a) $M$ has a maximal essential covering $\varepsilon:P\rightarrow M$, i.e. if $\rho:Q\rightarrow M$ is any essential covering of $M$ then there is a continuous homomorphism $\varphi:P\rightarrow Q$ such that $\varepsilon=\rho\circ\varphi$.
	
	(b) $P$ is unique up to isomorphism, i.e. if $\varepsilon_0:P_0\rightarrow M$ is another maximal essential covering of $M$ then we have a topological isomorphism $\alpha:P\rightarrow P_0$ such that $\varepsilon=\varepsilon_0\circ\alpha$.
	
	(c) $P$ is a projective object of $\mathcal{C}$, and $\varepsilon:P\rightarrow M$ is the smallest projective covering of $M$, i.e. if $\rho:Q\rightarrow M$ is a continuous epimorphism and $Q$ is a projective object of $\mathcal{C}$, then there a continuous homomorphism $\psi:Q\rightarrow p$ such that $\rho=\varepsilon\circ\psi$.

	The unique $P$ satisfying the above properties is called the \textbf{projective envelope}
	of $M$.
\end{definition}	
	
	\begin{prop}
	The projective envelope of $R/\mathtt{J}$ in $\mathcal{C}$ is $\widehat{R}$.
\end{prop}
\begin{proof}Note that $\widehat{R}$ is linearly compact for it is the inverse limit of the artinian modules $R/\mathtt{J}^n$.
	
	The natural homomorphism $\widehat{R}\rightarrow R/\mathtt{J}$ is an essential covering of $R/\mathtt{J}$. For suppose $N$ is a closed $R$-submodule of $\widehat{R}$ and $N+\widehat{\mathtt{J}}=\widehat{R}$. Since $N$ is closed it is an ideal of $\widehat{R}$; since $\widehat{(\cap_{i=1}^n\mathfrak{m}_i)}=\cap_{i=1}^n\widehat{\mathfrak{m}_i}$ by [Kle] 22.53, we must have $N=\widehat{R}$.
	
	We claim that $\widehat{R}$ is a projective object of $\mathcal{C}$ by [Mac] 5.12. If $M$ is a quotient module of $R^*$, $M$ is an $\widehat{R}$-module by Remark 1.5. Let $f:\widehat{R}\rightarrow M$ be any continuous $R$-homomorphism. Then $f$ is an $\widehat{R}$-homomorphism, for if $f(1)=\overline{m}\in M$ and $\widehat{x}\in\widehat{R}$, let $(x_n)$ be a Cauchy sequence in $R$ whose limit is $\widehat{x}$, then $f(x_n)=x_n\overline{m}$ for all $n$, and $\widehat{x}\,\overline{m}=x_n\overline{m}$ for all sufficiently large $n$. Hence $f(\widehat{x})=limf(x_n)=\widehat{x}\,\overline{m}$. 
	
	The diagram
	
	$$\xymatrix@C=12pt{ &&&&R^*\ar[rr]&&M\ar[rr]&&0&&\\&&&&&\widehat{R}\ar[ur]_{f}&&&&&}
	$$
	is therefore a diagram of $\widehat{R}$-modules and $\widehat{R}$-homomorphisms. $f$ can therefore be lifted to an $\widehat{R}$-homomorphism $g:\widehat{R}\rightarrow R^*$. $\widehat{R}/ker\,g$ is artinian, hence $ker\,g\supseteq\widehat{\mathtt{J}}^n$ for some $n$, hence $ker\,g$ is open and therefore $g$ is continuous. Hence by [Mac] 5.12  $\widehat{R}$ is a projective object of $\mathcal{C}$. The result follows from Definition 1.9.
\end{proof}		
	
	Here it was necessary to use $\widehat{(\cap_{i=1}^n\mathfrak{m}_i)}=\cap_{i=1}^n\widehat{\mathfrak{m}_i}$.
	
		

\section{J-Separated Modules and Macdonald Duality}
	In this section $(R,\,\mathfrak{m}_1,\,...,\,\mathfrak{m}_n)$ shall be a \textit{complete}, with the $\mathtt{J}$-adic topology, (noetherian semilocal, as usual) ring, where $\mathtt{J}=\bigcap_{i=1}^n\mathfrak{m}_i$.

	In this setting, for $i\neq j$, $\mathfrak{m}^r_i$ and $\mathfrak{m}^r_j$ are coprime for all $i,j=1,...,\,n$ and $r>0$ (see [A-M]  1.16). Thus, by the Chinese Remainder Theorem we have that $$\dfrac{R}{\mathtt{J}^r}\cong\dfrac{R}{\mathfrak{m}_1^r}\oplus...\oplus\dfrac{R}{\mathfrak{m}_n^r}.$$ Hence $ \dfrac{R}{\mathtt{J}^r}$ is artinian for all $r>0$. 
	
	Since $R$ is complete it is a linearly compact $R$-module by [Mac] 5.5, because $R\cong\varprojlim_{n}\dfrac{R}{\mathtt{J}^n}$, it is an inverse limit of artinian modules. If $M$ is an $R$-module, which is \textbf{Hausdorff} in the $\mathtt{J}$-adic topology, that is, $\bigcap_{n}\mathtt{J}^nM=0$, we say that $M$ is $\mathtt{\textbf{J}}$-\textbf{separated}.
	

	\begin{prop}
	
		If $M$ is artinian and $\mathtt{J}$-separated then $M$ is of finite length.
	\end{prop}
	\begin{proof} Since is artinian and $\cap\mathtt{J}^nM=0$, we must have $\mathtt{J}^nM=0$ for some $n\geq 0$. Let $D(M)=\text{ }^AHom_R(M,\,R^*)$ be the $R$-module of all homomorphisms (not necessarily continuous) from $M$ into $R^*$. Since $M$ is artinian, $D(M)$ is noetherian by [Ooi] 1.6 (3) and $\mathtt{J}^nD(M)=0$. Suppose $D(M)$ is generated by $y_1,\dots,y_r$. Then each $Ry_i$ is annihilated by $\mathtt{J}^n$, hence is of finite length hence $\displaystyle D(M)=\sum_{i=1}^rRy_i$ is of finite length. Hence by [Ooi] 1.6 (2),(3) and (5) $M$ is of finite length.
	\end{proof}	
In the proof above it was necessary to use Matlis dual results obtained by A. Ooishi for the semilocal case instead of results obtained originally by E. Matlis for the local case.

	\begin{prop}
	
	Let $M$ be an $R$-module. If $M$ is linearly compact, semi-discrete and $\mathtt{J}$-separated then $M$ is noetherian.
\end{prop}
\begin{proof}
	
	 Since $M$ is linearly compact and semi-discrete there is $N$ a finitely generated submodule of $M$ such that $M/N$ is artinian module, by [Z{\"o}s] Theorem. Since $R$ is noetherian we have that $N$ is noetherian. By [Mac] 7.1 and Proposition 2.1 $M/N$ is of finite length. Thus, $M/N$ is noetherian and considering the exact sequence \[0\rightarrow N\rightarrow M\rightarrow M/N\rightarrow0\]follows that $M$ is noetherian.
\end{proof}	
In the original proof I. G. Macdonald works with a module which is a finite-dimensional $R/\mathfrak{m}$ vector space, which is not the case in our setting.

	
		
		
		
	
	

	If $M$ is any $R$-module, then the \textbf{Macdonald Dual} of $M$ is \newline $M^*=Hom_R(M,\,R^*)$, the module of all continuous homomorphisms of $M$ into $R^*$,  the \textbf{character-module} of $M$, topologized as in [Mac] 8.1.     
	
	If $S$ is any subset of $M$, $S^0$ denotes the annihilator of $S$ in $M^*$, that is, the set of all $u\in M^*$ such that $u(S)=0$. If $T$ is any subset of $M^*$, $^0T$ denotes the annihilator of $T$ in $M$, that is, the set of all $x\in M$ such that $u(x)=0$ for all $u\in T$. Note that $S^0$ is a submodule of $M^*$ and $^0T$ is a submodule of $M$.
	
	Since $R^*$ is discrete, a fundamental system of neighborhoods of zero of $M^*$ can be formed by the submodules $C^0$ where $C$ runs through the set of linearly compact $\mathtt{J}$-separated submodules of $M$. 
	


	\begin{remark} \textup{Let $M$ be an $R$-module. Each $x\in M$ defines a homomorphism $\omega(x):M^*\rightarrow R^*$ by the rule $\omega(x)(u)=u(x)$. The kernel of $\omega(x)$ is $<x>^0$, which is open in $M^*$ since $<x>$ is linearly compact and $\mathtt{J}$-separated and therefore $\omega(x)$ is a character of $M^*$. Hence we have a homomorphism $$\omega:M\rightarrow M^{**}.$$}
		\textup{If $\omega(x)=0$ then $u(x)=0$ for all $u\in M^*$, hence $x=0$ by Proposition $2.5$ and consequently $\omega$ is an (algebraic) monomorphism. We will characterize those $R$-modules for which $\omega$ is a topological isomorphism. }  
	\end{remark}
	\begin{prop} 
		Let $M$ be a linearly compact $\mathtt{J}$-separated $R$-module and let $N$ be a submodule of finite length in $M^*$. Then $(^0N)^0=N$.
	\end{prop}
	\begin{proof} Let $n$ be the length of $N$. We will use induction on $n$. Consider first $n=1$. Then $N\cong R/\mathfrak{m}_j=K_j$ for some $j$. Indeed, since $N$ is simple, there is $x\in N$ such that $a\mapsto ax$, where $a\in R$, induces the isomorphism $R/Ann(x)\cong N$ and $R/Ann(x)$ is a simple $R$-module, hence it is a field and $Ann(x)=\mathfrak{m}_j$ for some $j$. Now, if $x\in M$ then $\omega(x)|N$ is a character of $N$, hence $x\mapsto\omega(x)|N$ is a homomorphism $M\rightarrow N^*$ whose kernel is $^0N$. It follows that $M/^0N$ is isomorphic to a submodule of $N^*$. Since $N\cong K_j$ and by [Ooi] 1.1 (2) we have that $K_j\cong K^*_j$ then  $N^*\cong K_j$ and therefore $M/^0N\cong K_j$ (otherwise we should have $M=\text{}^0N$ and thus $N=0$). Hence $(M/^0N)^*\cong K_j$ and therefore by [Mac] 9.2 $(^0N)^0\cong K_j$. Since $N\subseteq(^0N)^0$ we have that $N=(^0N)^0$. Now suppose the result is true for all submodules os length $<n$ in $M^*$ ($n>1$). If $N$ is a submodule of length $n$ in $M^*$ we have that it has a submodule $P$ of length $n-1$. By the inductive hypothesis $(^0P)^0=P$ and hence $(^0P)^*\cong M^*/P$ by [Mac] 9.2. If we identify $(^0P)^*$ and $M^*/P$ by this isomorphism, the annihilator $^0(N/P)$ of $(N/P)$ in $^0P$ is $^0N$. Since $N/P$ is of length $1$ and $^0P$ is linearly compact and $\mathtt{J}$-separated (since $^0P$ is closed in $M$, by [Mac] 9.1) we have that, by the case $n=1$, $(^0(N/P))^0=N/P$. But $(^0(N/P))^0=(^0N)^0/P$ and therefore $(^0N)^0=N$. 
	\end{proof}
In the above proof it was necessary to use Matlis duality obtained by A. Ooishi in the semilocal case and that $N\cong R/\mathfrak{m}_j=K_j$ for some $j$.
	\begin{prop} 
	If $M$ is linearly compact and $\mathtt{J}$-separated then $\omega(M)=M^{**}$.
	\end{prop}
	\begin{proof} Let $N$ be a noetherian submodule of $M^*$. By [Mac] 9.4 $M^*$ is discrete and hence $N$ is discrete. Furthermore $N$ is linearly compact by [Mac] 7.3 and it is artinian by Theorem 1.7 and [Mac] 5.4. So $N$ is of finite length and therefore $N=(^0N)^0\cong(M/ ^0N)^*$ by Proposition 2.4 and [Mac] 9.2. Consider the homomorphism $\omega_N:M\rightarrow N^*$ given by $x\mapsto\omega(x)|_{N}$, then $Ker(\omega_N)=$ $^0N$ and hence $(M/^0N)$ is a submodule of $N^*$ which is artinian by [Ooi] 1.6 (2). Hence, by [Ooi] 1.6 (5), $(M/^0N)^{**}\cong M/^0N\cong N^*$. Also $\omega_N$ is an epimorphism, that is, if $u\in M^{**}$ then $u|_N\in N^*$ and there is $x_N\in M$ such that $\omega(x_N)|_N=u|_N$. Consider the cosets  $x_N+\,^0N$ of $M$. As $N$ runs through the noetherian submodules of $M^*$ these are closed cosets by [Mac] 9.1 and they have the finite intersection property. Indeed, for if $N_1,\dots,N_r$ are noetherian submodules of $M^*$ then $N_1+\dots+N_r$ is noetherian, hence there is $x\in M$ such that $\omega(x)|_{N_1+\dots+N_r}=u|_{N_1+\dots+N_r}$. For each $i$ there is $x_{N_i}\in M$ such that $\omega(x_{N_i})|_{N_i}=u|_{N_i}$, also $\omega(x)|_{N_i}=u|_{N_i}$, since $N_i\subseteq N_1+\dots+N_r$. Hence $\omega(x)|_{N_i}=\omega(x_{N_i})|_{N_i}$ and $x-x_{N_i}\in Ker(\omega_{N_i})=$ $^0N_i$, that is, $\displaystyle x\in\bigcap_{i=1}^r(x_{N_i}\,+^0N_i)$.
		
		Since $M$ is linearly compact it follows that there is an element $y$ of $M$ such that $y\in x_N+\,^0N$ for all noetherian submodules $N$ of $M^*$. Hence $\omega(y)$ agrees with $u$ on every noetherian submodule of $M^*$ and therefore $\omega(x)=u$. Thus $\omega$ is an epimorphism. 	  
	\end{proof} 
	To prove the above Proposition we use duality and reflexivity results obtained by A. Ooishi. Furthermore, it was important to realize that the Matlis dual and the Macdonald dual coincide when the module is semi-discrete, which is true for noetherian modules and artinian modules by Proposition 2.2 and [Zel] Proposition 5.

By means of the semilocal tools constructed in this work, we establish the validity of Macdonald's main duality theorem [Mac] in the semilocal setting.

	\begin{prop} 	Let $(R,\,\mathfrak{m}_1,\,...,\,\mathfrak{m}_n)$ be a semilocal $\mathtt{J}$-adically complete ring, where $\mathtt{J}=\bigcap_{i=1}^n\mathfrak{m}_i$.
		If $\{M_i\}$ is a family of linearly compact $R$-modules, we have topological isomorphisms
		$$(\prod_{i\in I}M_i)^*\cong\bigoplus_{i\in I}M_i^*,$$
		$$(\varprojlim M_i)^*\cong\varinjlim M_i^*.$$
	\end{prop}

\section{Applications to Local Homology and Local cohomology}


We recall the definition of Local Homology of [C-N2]. Let $I$ be an ideal of $R$, the \textit{i th local homology} modules $H^{I}_i(M)$ of an $R$-module $M$ with respect to $I$ is defined by (see [C-N2]  3.1)\[H^I_{i}{(M)}=\varprojlim_{t} \Tor^R_{i}(R/I^t,\,M).\]
Note that $H^{I}_0(M)\cong\Lambda_I(M)$, in which $\Lambda_I(M)=\varprojlim_tM/I^tM$ the $I$-adic completion of $M$.
Let $M$ be a linearly compact $R$-module. Then $Tor^R_i(R/I^t,\,M)$ is also a linearly compact
$R$-module by the topology defined as in [C-N1] 2.6, so we have an induced topology on the local homology
module $H^I_i(M)$.

Using Theorem 3.8 of [C-N1] we can prove the following criterion for a finitely generated module on a semilocal Noetherian ring to be linearly compact. 

\begin{prop} Let $(R,\,\mathfrak{m}_1,\,...,\,\mathfrak{m}_n)$ be a semilocal noetherian ring and $M$ a finitely generated $R$-module. Then $M$ is a linearly compact $R$-module if and only if $M$ is complete with respect to the $\mathtt{J}$-adic topology, where $\mathtt{J}=\bigcap_{i=1}^n\mathfrak{m}_i$.
\end{prop}
\begin{proof} Since $M$ is a finitely generated $R$-module, $M$ is $\mathtt{J}$-separated by Krull's intersection theorem ([G-S] 1.5). Thus, if $M$ is a linearly compact $R$-module, $\Lambda_{\mathtt{J}}(M)\cong M$ by [C-N1] 3.8. Conversely, if $M$ is complete in $\mathtt{J}$-adic topology, we have $M\cong\varprojlim_tM/\mathtt{J}^tM$. Therefore is a linearly compact $R$-module by [Mac] 3.7, as $M/\mathtt{J}^tM$ are artinian $R$-modules for all $t>0$, since $M/\mathfrak{m}_i^tM$ are artinian $R$-modules for all $t>0$ and for all $i=1,\dots,\,n$ and by Chinese Remainder Theorem.
\end{proof}

Denote by $L(M)$ the sum of all artinian submodules of $M$.
 The $I$-torsion functor $\Gamma_I$ is defined by $\Gamma_I(M)=\bigcup_{t>0}(0:_M\,I^t)$.

The following result is proved for the local case in [C-N1] 4.5 as an immediate consequence of the other two results. Here, for the semilocal case, a different proof was needed.

\begin{prop}
	Let $(R,\,\mathfrak{m}_1,\,...,\,\mathfrak{m}_n)$ be a semilocal noetherian ring, $\mathtt{J}$ the Jacobson radical of $R$ and $M$ a semi-discrete linearly compact $R$-module. Then \[L(M)=L(\Gamma_{\mathtt{J}}(M))\]and\[H_i^{\mathtt{J}}(M)\cong H_i^{\mathtt{J}}(\Gamma_{\mathtt{J}}(M))\]for all $i>0$.
\end{prop}
\begin{proof} By [B-S] 1.1.2, we have $\Gamma_{\mathfrak{m}_{j}}(\Gamma_{\mathtt{J}}(M))=\Gamma_{\mathfrak{m}_j}(M)$ for each $j=1,\,\dots,\,n$. Since $\Gamma_{\mathfrak{m}_j}(M)\neq0\Leftrightarrow\Gamma_{\mathfrak{m}_{j}}(\Gamma_{\mathtt{J}}(M))\neq0$, follows by [Sha] 1.4 and [C-N1] 4.3 that $L(M)=L(\Gamma_{\mathtt{J}}(M))$. Furthermore, from proof of [C-N1] 4.4, follows that $H_i^{\mathtt{J}}(M)\cong H_i^{\mathtt{J}}(\Gamma_{\mathtt{J}}(M))$.
\end{proof}

\begin{prop} 
		
	 Let $M$ be an $R$-module. Then for all $i\geq0$,\[L_i^{I}(D(M))\cong D(H^i_I(M))\cong H^I_i(D(M)).\]

\end{prop}
\begin{proof} From  [Mac] 2.6 and [Ann] 4.1, \[\begin{array}{rcl}
		H^I_i(D(M))&\cong&\displaystyle\varprojlim_{t}\Tor^R_i(\frac{R}{I^t};D(M))\\&\cong&\displaystyle\varprojlim_{t}D(\Ext^i_R(\frac{R}{I^t};M))\\&\cong&D(\displaystyle\varinjlim_{t}\Ext^i_R(\frac{R}{I^t};M))\\&=&D(H^i_I(M)).
	\end{array}\]The result is completed by [Ann] 5.6.
\end{proof}
By Proposition 3.3 and [C-N2] 6.4, we have:
\begin{cor}
	Let $(R,\,\mathfrak{m}_1,\,...,\,\mathfrak{m}_n)$ be a complete semilocal noetherian ring.
\end{cor}

\begin{enumerate}
	\item [$(i)$] If $M$ is linearly compact $R$-module, then for all $i\geq0$,$$H^{i}_I(M)\cong(H^{I}_{i}(M^*))^*,$$
	$$H^{I}_i(M)\cong(H^{i}_{I}(M^*))^*.$$
	
	\item [$(ii)$] If $M$ is a semi-discrete linearly compact $R$-module, then we have topological isomorphism of $R$-modules for all $i\geq0$,
	$$H^{i}_I(M)\cong(H^{I}_{i}(M^*))^*,$$
	$$H^{I}_i(M)\cong(H^{i}_{I}(M^*))^*.$$
	
\end{enumerate}

\section{Dimension, Vanishing and non-vanishing}


Remember that the (\textit{Krull}) dimension $dim_RM$ of a non-zero $R$-module $M$ is the supremum of lengths of chains of primes in the support of $M$ if this supremum exists, and $\infty$ otherwise. If $M$ is finitely generated, then $dim M=$max$\{dim R/\mathfrak{p}\,\arrowvert\,\mathfrak{p}\in Ass M\}$. For convenience, we set $dim M=-1$ if $M=0$.
Note that if $A$ is a artinian module, then $dim (A)=0$, since support of $A$ consists of maximals only. 

We now recall the concept of \textit{noetherian dimension} of an $R$-module $M$ denoted by Ndim $M$.
Note that the notion of noetherian dimension was introduced first by R.N. Roberts [Rob] by the
name Krull dimension. Later, D. Kirby [Kir] changed this terminology of Roberts and refereed
to noetherian dimension to avoid confusion with well-known Krull dimension of finitely generated
modules. Let $M$ be an $R$-module. When $M=0$ we put $Ndim\,M=-1$. Then by induction,
for any ordinal $\alpha$, we put $NdimM=\alpha$ when (i) $NdimM <\alpha$ is false, and (ii) for every ascending
chain $M_0\subseteq M_1 \subseteq\cdots$ of submodules of $M$, there exists a positive integer $m_0$ such that
$Ndim(M_{m+1}/M_m) < \alpha$ for all $m\geq m_0$. Thus $M$ is non-zero and finitely generated if and only
if $Ndim\,M=0$. If $0\rightarrow M''\rightarrow M\rightarrow M'\rightarrow0$ is a short exact sequence of R-modules, then
$Ndim\,M=$ max$\{NdimM'',\,NdimM'\}$. 

\begin{remark}\text{(See [C-N1] 4.6).}

		 In case $M$ is an artinian module, $Ndim\,M<\infty$ (see [Kir] 2.6). More general, if $M$ is a semi-discrete linearly compact module, there is a short exact sequence $0\rightarrow N\rightarrow M\rightarrow A\rightarrow0$ where $N$ is finitely generated and $A$ is artinian (See [Z{\"o}s] Theorem). Hence $Ndim\,M=$ max$\{NdimM'',\,NdimM'\}<\infty$.
	
\end{remark}

The proof of the following proposition is similar to the original, however, in item $(ii)$, it is necessary to use [G-S] 4.6 and [G-S] 7.5, so that $dim_RM=dim_{\widehat{R}}M$ when $M$ is finitely generated.

\begin{prop} Let $M$ be a semi-discrete linearly compact $R$-module. Then
	\begin{enumerate}
		\item [$(i)$] $Ndim_RM=Ndim_{\widehat{R}}M$;
		\item [$(ii)$]
		$dim_RM=dim_{\widehat{R}}M$.
		
	\end{enumerate}
	
\end{prop}

\begin{remark}\text{ (See [C-N1] $7.3$)}
	\begin{enumerate}
		\item [(i)] Denote by $\mathcal{C}$ the category of semi-discrete linearly compact $R$-modules. It is well known
		that the category $\mathcal{C}$ contains the category of artinian $R$-modules and also the category of finitely
		generated $R$-modules if $R$ is complete. However, there are many semi-discrete linearly
		compact $R$-modules which are neither artinian nor finitely generated. An example for
		this conclusion is the module $K$ in [C-N1] 4.9. More general, let $R$ be complete ring, $A$ an
		artinian $R$-module with $Ndim\,A>0$ and $N$ a finitely generated $R$-module with $dim\,N >0$.
		Then $M=A\oplus N$ is semi-discrete linearly compact. Further, let $Q$ be a quotient module of
		$M$, then $Q$ is also a semi-discrete linearly compact R-module.
		\item [(ii)] If $M\in\mathcal{C}$, then by [Mac] 5.8 the Matlis dual $D(M)$ and the Macdonald dual $M^*$ are the same. Moreover,
		$M$ is linearly discrete by [C-N1] 6.8 and can be regarded by [C-N1] 7.1 as an $\widehat{R}$-module, therefore the
		Macdonald dual functor $(-)^*$ is a functor from $\mathcal{C}$ to itself and we have by [C-N1] 6.3 a topological
		isomorphism $\omega:M\rightarrow M^{**}$. Thus, $(-)^*$ is an equivalent functor on the category $\mathcal{C}$.
	\end{enumerate}
\end{remark}

\begin{prop} Let $M$ be a semi-discrete linearly compact R-module. Then
	\[Ndim_RM^*=dim_RM\;\; \text{and}\;\; Ndim_RM=dim_RM^*.\]
	
\end{prop}
\begin{proof} From [C-N1] 7.1 (ii) and Proposition 4.2 we may assume that \newline $(R,\,\mathfrak{m}_1,\,...,\,\mathfrak{m}_n)$ is a complete ring. If $M$ is finitely generated $R$-module, $M^*$ is artinian. Keep in mind in our case that $D(M)=M^*$. Let's show that  \newline$Ndim\,M^*=dim\,M$ when $M$ is finitely generated. Since $M^*$ is artinian, we have that $M^*=\bigoplus_{i=1}^rA_i$ where $A_i$ is a $R_{\mathfrak{m}_i}$-module for each $i$, by [Sha] 1.4 and 1.7. From [C-N1] 7.4, $Ndim\,A_i=dim\,A_i^*$. Thus,
	\begin{eqnarray*}
		Ndim\,M^* & = & max\{Ndim\,A_i;\,i=1,\dots\,,r\}\\
		& = & max\{dim\,A_i^*;\,i=1,\dots\,,r\}\\
		& = & dim\,M^{**}\\
		& = & dim\,\widehat{M}\\
		& = & dim\,M
	\end{eqnarray*}	where the last two equalities come from [Ooi] 1.6 (iv) and [C-N1] 3.8. If $M$ is artinian, then $Ndim\,M^*=dim\,M=0$.
	
	Suppose now that $M$ is semi-discrete linearly compact. There is by [Z{\"o}s] (Theorem) a short exact sequence $0\rightarrow N\rightarrow M\rightarrow A\rightarrow 0$ in which $N$ is finitely
	generated and $A$ is artinian. Thus we get by Macdonald duality an exact sequence $0\rightarrow A^*\rightarrow M^*\rightarrow N^*\rightarrow 0$, where $N^*$ is artinian and $A^*$ is finitely generated. Then\begin{eqnarray*}
		Ndim\,M^* & = & max\{Ndim\,N^*,\,Ndim\,A^*\}\\
		& = & max\{dim\,N,\,dim\,A\}\\
		& = & dim\,M.
	\end{eqnarray*}
	The second equality follows from Remark 4.3 (ii).
\end{proof}	

The following lemma was inspired by [LRN] 3.16.

\begin{lemma} Let $M$ be
	a non-zero artinian $R$-module with $Ndim\,M=d$ and $\mathtt{J}$ the Jacobson radical of $R$. Then $H^{\mathtt{J}}_d(M)\neq0$.
\end{lemma}
\begin{proof} From [C-N1] 7.1 (ii), Proposition 4.2 and Independence Theorem ([B-S] 4.2.1) we may assume that $R$ is a complete ring. By [Ooi] 1.6 the Matlis dual  $D(M)$ is a non-zero finitely generated and $dim\,D(M)=d$ by Proposition 4.4. Then $H^d_{\mathtt{J}}(D(M))\neq0$, otherwise, $H^d_{\mathfrak{m}R_{\mathfrak{m}}}(D(M)_{\mathfrak{m}})=H^d_{\mathtt{J}R_{\mathfrak{m}}}(D(M)_{\mathfrak{m}})=0$ for any maximal $\mathfrak{m}$ of $R$ and this contradicts Grothendieck's Non-Vanishing Theorem. Therefore, by Proposition 3.5 and [Ooi] 1.6 \[H_d^{\mathtt{J}}(M)\cong H_d^{\mathtt{J}}D(D((M)))\cong D(H^d_{\mathtt{J}}(D(M)))\neq0\]and this prove the result.\end{proof}

In [C-N1] 4.10 it was shown that $Ndim\,M=\,\text{max}\,\{i\,|\,H^{\mathfrak{m}}_i(M)\neq0\}$ if $Ndim\,M\neq1 $ when $(R,\,\mathfrak{m})$ is a local noetherian ring and $M$ is a non-zero semi-discrete linearly compact module. We generalize this result for semilocal case. Item $(i)$ required changes to the proof.
In [C-N1] 4.9 there is counterexample when $Ndim\,M=1$.

\begin{prop} Let $R$ be a semilocal noetherian ring,  $\mathtt{J}$ the Jacobson radical of $R$ and $M$ a non-zero semi-discrete linearly compact $R$-module. Then
	\begin{enumerate}
		\item [(i)] $Ndim\,\Gamma_{\mathtt{J}}(M)=$ max$\{i\,|\,H^{\mathtt{J}}_i(M)\neq0\}$ if $\Gamma_{\mathtt{J}}(M)\neq0$;
		\item [(ii)] $Ndim\,M=$ max$\{i\,|\,H^{\mathtt{J}}_i(M)\neq0\}$ if $Ndim\,M\neq1$.
	\end{enumerate}
\end{prop}
\begin{proof} $(i)$ By Theorem 1.7 $\Gamma_{\mathtt{J}}(M)$ is artinian $R$-module. From [C-N1] 4.8 and Lemma 4.5 we have that\[Ndim\,\Gamma_{\mathtt{J}}(M)=\, \text{max}\{i\,|\,H^{\mathtt{J}}_i(\Gamma_{\mathtt{J}}(M))\neq0\}.\] The result follows by Proposition 3.2.
\end{proof}

In [C-N1] 4.12 there is a example which shows that the condition $\Gamma_{\mathtt{J}}(M)\neq0$ in Proposition 4.6 (i) is needful.
The proof of the following corollary is similar to the original, however, it is necessary to use Lemma 4.5.

\begin{cor} Let $R$ be a semilocal noetherian ring, $\mathtt{J}$ the Jacobson radical of $R$ and $M$ a non-zero semi-discrete linearly
	compact module. Then $H^{\mathtt{J}}_i(M)=0$ for all $i\geq0$ if and only if there exists an element $x\in\mathtt{J}$ such
	that $xM=M$ and $0:_Mx=0$.
\end{cor}

In [C-N1] 7.5, Cuong and Nam extend well-known results in Grothendieck's local cohomology theory of
finitely generated R-modules for semi-discrete linearly compact modules. We generalized to the semilocal case. In order for a proof analogous to the original case to be possible, it is necessary to use results 1.8, 4.2 (ii), 4.4 and 4.6 (i), all obtained by us and which required changes in their proofs.

\begin{prop}Let $R$ be a semilocal noetherian ring,  $\mathtt{J}$ the Jacobson radical of $R$ and $M$ a non-zero semi-discrete linearly compact $R$-module. Then
	\begin{itemize}
		\item [(i)] $dim_R\,M=$ max$\{i\,|\,H_{\mathtt{J}}^i(M)\neq0\}$ if $dim_R\,M\neq1$;
		\item [(ii)] $dim_R\,(\widehat{M})= $max$\{i\,|\,H_{\mathtt{J}}^i(M)\neq0\}$ if $\widehat{M}\neq0$.
	\end{itemize}
\end{prop}

In [C-N1] 7.6 there is examples which show that the conditions $dim\,M\neq1$ and $\widehat{M}\neq0$ in Proposition 4.8 are needful. 

The next four results mentioned are obtained with analogous proofs, of course, replacing the results originally used with our results. It can be concluded that all results from [C-N1] were generalized from local case to semilocal case.

\begin{cor} Let $R$ be a semilocal noetherian ring, $\mathtt{J}$ the Jacobson radical of $R$ and $M$ a non-zero semi-discrete linearly
	compact module. Then $H_{\mathtt{J}}^i(M)=0$ for all $i\geq0$ if and only if there exists an element $x\in\mathtt{J}$ such
	that $xM=M$ and $0:_Mx=0$.
\end{cor} 

Recall that a sequence of elements $x_1,\,\dots,\,x_r$ in $R$ is said to be an $M$-\textit{regular} sequence if $M/ (x_1,\,\dots,\,x_r)M\neq0$ and $M/ (x_1,\,\dots,\,x_r)M\xrightarrow{x_i}M/ (x_1,\,\dots,\,x_r)M$ is injective for
$i=1,\,\dots,\,r$. Denote by $depth_I\,(M)$ the supremum of the lengths of all maximal $M$-regular
sequences in the ideal $I$. Then we have the following result.

\begin{prop} Let $R$ be a semilocal noetherian ring and $M$ a semi-discrete linearly compact $R$-module and $I$ an ideal of $R$ such that $M/IM\neq 0$. Then \[depth_I\,(M)=\,inf\{\,i\,|\,H^{i}_I(M)\neq0\}.\]
\end{prop}

Following is the artinianness of local cohomology modules.

\begin{prop} Let $R$ be a semilocal noetherian ring, $\mathtt{J}$ the Jacobson radical of $R$ and $M$ a semi-discrete linearly compact $R$-module with \newline $dim_RM=d$. Then the following statements are true.
	\begin{enumerate}
		\item [$(i)$] The local cohomology modules $H_{\mathtt{J}}^{i}(M)$ are artinian $R$-modules for all $i\geq0$.
		\item [$(ii)$] The local cohomology module $H^d_I(M)$ is artinian. 
	\end{enumerate}
\end{prop}

\begin{cor} Let $R$ be a semilocal noetherian ring, $\mathtt{J}$ the Jacobson radical of $R$ and $M$ a finitely generated $R$-module with $dim_RM=d$. Then the local cohomology modules $H_{\mathtt{J}}^{i}(M)$ and $H_I^d(M)$ are artinian $R$-modules for all $i\geq0$.
\end{cor}






 Departamento de Ciência e Tecnologia, Universidade Federal Rural do Semi-árido - 59780-000, Caraúbas, RN, Brazil.

\end{document}